\documentclass[a4paper,11pt,reqno,pdftex]{amsart}

\usepackage{amsmath}
\usepackage{amssymb}
\usepackage{graphicx}
\usepackage{xcolor}
\usepackage{enumerate}

\usepackage{mathrsfs}
\renewcommand{\mathcal}{\mathscr}

\theoremstyle{definition} 
\newtheorem{theorem}{Theorem}[section]

\newtheorem{lemma}[theorem]{Lemma}

\newtheorem{remark}[theorem]{Remark}

\newtheorem*{thm*}{Theorem}

\makeatletter
    
    \@addtoreset{equation}{section}
  \makeatother

\title[Elephant random walks remembering the very recent past]{Limit theorems for elephant random walks remembering the very recent past, with applications to the Takagi--van der Waerden class functions}
\thanks{M.T. is partially supported by JSPS KAKENHI Grant Numbers JP19H01793, JP19K03514 and JP22K03333.
}
\author{Yuzaburo Nakano}
\address{Graduate School of Engineering Science, Yokohama National University, Yokohama, Japan}
\email{nakano-yuzaburo-zg@ynu.jp}
\author{Masato Takei}
\address{Department of Applied Mathematics, Faculty of Engineering, Yokohama National University, Yokohama, Japan}
\email{takei-masato-fx@ynu.ac.jp}

\begin{document}

\begin{abstract}
We study the Takagi--van der Waerden functions $f_r(x)$, a well-known class of continuous but nowhere differentiable functions, from probabilistic point of view. As an application of elephant random walks remembering the very recent past (ERWVRP, a.k.a. symmetric correlated random walks), we obtain precise estimates for the oscillations of $f_r(x)$. We also establish a result on the necessary and sufficient condition for localization of the ERWVRP with variable step length, which can be applied to obtain a complete description of the differentiability properties of the Takagi--van der Waerden class functions.

\end{abstract}

\maketitle

\section{Introduction}

The Takagi--van der Waerden functions are defined by 
    \begin{equation}\label{eq:TakagivanderWaerdenfunction}
        f_{r}(x) :=  \sum_{k=1}^{\infty} \frac{1}{r^{k-1}}d(r^{k-1}x)
    \end{equation}
    for all integer \(r\ge 2\), where \(d(x)\) denotes the distance from \(x\) to its nearest integer, i.e.
    \begin{equation}\label{eq:distancefrominteger}
        d(x) :=  \min_{n\in \mathbb{Z}}|x-n|.
    \end{equation}
    They constitutes a well-known class of continuous nowhere differentiable functions: 
    The function \(f_{2}\) is first studied by Takagi \cite{Takagi1903}, while \(f_{10}\) is introduced apparently independently by van der Waerden \cite{vanderWaerden30}. A short proof of the nondifferentiability of \(f_{r}\) is found in Allaart \cite[Theorem 2.1]{Allaart14JMAA}. For surveys and more references on the Takagi function and its relatives,  see \cite{AllaartKawamura11/12,JarnickiPflug2015,Lagarias12}. 
    
In his highly influential work \cite{Kono87ActaMathHungar}, K\^{o}no analyzed the several properties of the Takagi function and its generalizations, from probabilistic point of view. This approach is quite powerful, and after \cite{Kono87ActaMathHungar} many works in this direction appeared (See \cite{Gamkrelidze90,Allaart09JMSJ,Allaart14JMAA,deLimaSmania19,OsakaTakei20,HanSchiedZhang21} among others, and references therein).

\begin{figure}[t]
\label{fig:TakagivdWf2f3}
\begin{center}
\begin{tabular}{cc}
\includegraphics[width=5cm]{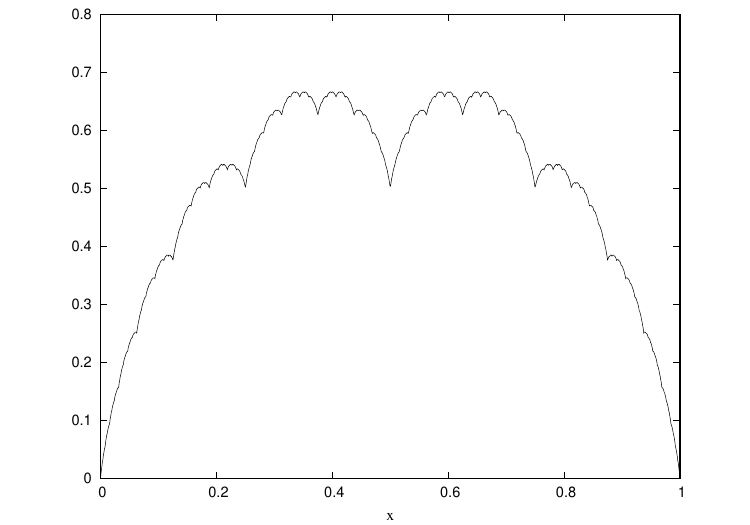}
&
\includegraphics[width=5cm]{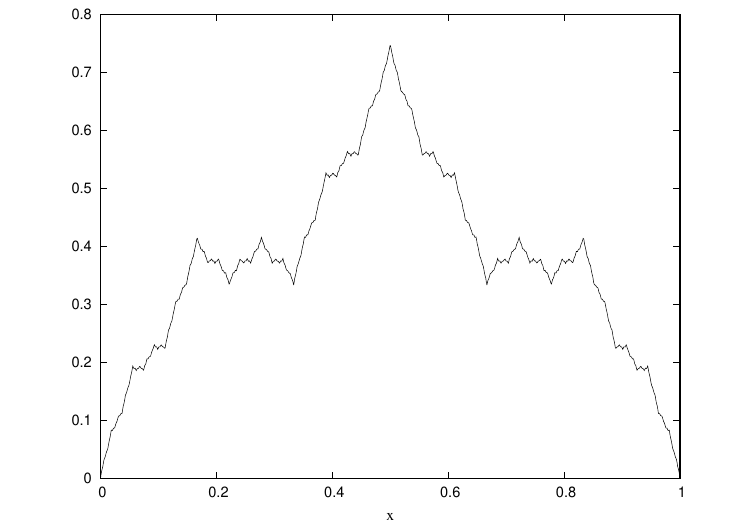}
\end{tabular}
\end{center}
\caption{The Takagi--van der Waerden functions $f_2(x)$ (left) and $f_3(x)$ (right).}
\end{figure}
    
As shown in Figure \ref{fig:TakagivdWf2f3}, the graphical properties of $f_r$ heavily depend on the parity of $r$. An essential difference between even $r$ and odd $r$ is well explained by the probabilistic approach first developed by Allaart \cite{Allaart14JMAA}.
We define a sequence of functions \( \{ \psi_{k}(x) : k \geq 1\} \) by 
    \begin{equation}\label{eq:psidefinition}
        \psi_{k}(x) :=  \frac{1}{r^{k-1}}d(r^{k-1}x)\quad \mbox{for $k \geq 1$.}
\end{equation}
Let \({\psi}_{k}^{+}(x)\) be the right-hand derivative of \({\psi}_{k}\) at \(x\), and we regard $\{ {\psi}_{k}^{+}(x) :k \geq 1\}$ as a \(\{+1,-1\}\)--valued stochastic process
on the probability space \(([0,1), \mathfrak{B}, \mu)\), where \(\mathfrak{B}\) is the Borel \(\sigma\)-field on \( [0,1)\), and \(\mu\) is the Lebesgue measure on \( [0,1) \).
Allaart \cite[Lemma 3.2]{Allaart14JMAA} showed that 
\(\{{\psi}_{k}^{+}(x) : k\ge 1 \}\) is a Markov chain with initial distribution 
\begin{equation}
    \mu({\psi}_{1}^{+}=+1)=\mu({\psi}_{1}^{+}=-1)=\frac{1}{2},
\end{equation}
and for every \(k\ge 1\),
\begin{gather}
\begin{cases}
    \mu({\psi}_{k+1}^{+}=+1\mid {\psi}_{k}^{+}=+1)=1-\mu({\psi}_{k+1}^{+}=-1\mid {\psi}_{k}^{+}=+1)=p_{r},\\
    \mu({\psi}_{k+1}^{+}=-1\mid {\psi}_{k}^{+}=-1)=1-\mu({\psi}_{k+1}^{+}=+1\mid {\psi}_{k}^{+}=-1)=p_{r},
\end{cases}\end{gather} 
where
\begin{equation}
        p_{r}=
        \begin{cases}
            \dfrac{1}{2} & \text{if $r$ is even,}\\
            \dfrac{r+1}{2r} & \text{if $r$ is odd.}
        \end{cases}
 \end{equation}
We set
\begin{equation}\label{eq:definitionrandomwalk}
   s_{0}(x)\equiv 0, \quad \mbox{and} \quad  s_{n}(x) :=  \sum_{k=1}^{n}{\psi}_{k}^{+}(x) \quad \mbox{for $n \ge 1$}.
\end{equation}
In \cite{Allaart14JMAA}, the stochastic process \(\{s_{n}(x) : n\ge 0 \}\) on the probability space \(([0,1), \mathfrak{B}, \mu)\) is called the {\it correlated random walk} (CRW) with parameter $p_r$. When \(r\) is even it is nothing but the symmetric simple random walk, while the increments of the walker have correlations when $r$ is odd. 

Recently there has been a lot of interest in the elephant random walk (ERW), introduced by Sch\"{u}tz and Trimper \cite{SchutzTrimper04}, and its variations. It is a nearest-neighbour random walk on the integers. Each time after the first step, the elephant selects a step from its past history (in the standard ERW in \cite{SchutzTrimper04} the elephant selects a step uniformly at random), and repeats the remembered step with probability $p$ or makes a step in the opposite direction with probability $1-p$.  An interesting feature of the standard ERW is phase transition at $p=3/4$ from diffusive regime to superdiffusive regime. See Laulin \cite{Laulin22PhD} for survey of this research field and further references. Gut and Stadtm\"{u}ller \cite{GutStadtmuller21JAP} initiated the study of several variations of the ERW, where the elephant selects a step not necessarily uniformly over the whole history. The research in this direction is developing very rapidly, and many papers including \cite{BenArietal21Brazil,GutStadtmuller22SPL,Laulin22ECP,ChenLaulin23,AguechElMachkouri24,RoyTakeiTanemura24a,RoyTakeiTanemura24b,BertenghiLaulin25} have appeared in recent years. As is pointed out in \cite[Sections 2 and 7]{GutStadtmuller21JAP}, the ERW {\it remembering the very recent past} (ERWVRP for short) coincides with the CRW. 
In retrospect, the traditional name ``correlated random walk" makes more or less an ambiguous impression, and in view of increasing interest on the ERW,  hereafter we use the ``ERWVRP" rather than the ``CRW".




The aim of the current paper is to explore several properties of the Takagi--van der Waerden class functions and their weighted version, as applications of limit theorems for the ERWVRP, possibly with variable step length. Our result on the local modulus of continuity of $f_r(x)$ (Theorem \ref{thm:main_theorem}) reflects how the nature of the oscillations of the corresponding ERWVRP depends on the parity of $r$ (Theorem \ref{thm:CorrRWCLTLIL}). Generalizing $f_r(x)$, we consider its weighted version
\begin{align} \label{def:TakagivdWclass}
f_{r,a}(x)  :=   \sum_{k=1}^{\infty} a_k \psi_k(x) = \sum_{k=1}^{\infty} \dfrac{a_k}{r^{k-1}} d(r^{k-1} x),
\end{align}
where $\{ a_k : k \geq 1 \}$ is a real sequence satisfying  
\begin{align} \label{def:TakagivdWclassl1}
\sum_{k=1}^{\infty} \left| \dfrac{a_k}{r^{k-1}} \right| <+\infty.
\end{align}
We call $f_{r,a}(x)$ the {\it Takagi--van der Waerden class function}. To classify its differentiability properties, 
we introduce the ERWVRP with variable step length, and establish a criterion for the localization of the elephant (Theorem \ref{thm:CorrRWRademacherKhintchineKolmogorov}). This result is quite different from the result on the standard ERW with variable step length \cite{Nakano25+}, and is of independent interest.
Also it enables us to obtain 
a complete description of the differentiability properties of $f_{r,a}(x)$ (Theorem \ref{thm:DIffTakagivdW}), which was left open for a long time when $r$ is odd.

The rest of the paper is organized as follows: We present our main results in Section \ref{sec:results}.
Limit theorems for the ERWVRP are proved in Section \ref{Sec:ProofERWRP}. Sections \ref{sec:ProofTakagiA} and \ref{sec:ProofTheoremDIffTakagivdW} are devoted to the proofs of our results on the Takagi-van der Waerden class functions.

\section{Results}
\label{sec:results}

\subsection{Limit theorems for the elephant random walk remembering the very recent past}

In this section the basic probability space is $(\Omega,\mathcal{F},P)$, and the expectation under $P$ is denoted by $E$.
Let $p \in (0,1)$,  and $\{X_k : k \geq 1\}$ be a $\{+1,-1\}$-valued Markov chain with
\begin{align*}
P(X_1=+1)=P(X_1=-1)=\dfrac{1}{2},
\end{align*}
and
\begin{align*}
\begin{cases}
P(X_{k+1}=+1 \mid X_k=+1)=P(X_{k+1}=-1 \mid X_k=-1)=p \\
P(X_{k+1}=-1 \mid X_k=+1)=P(X_{k+1}=+1 \mid X_k=-1)=1-p \\
\end{cases}
\mbox{for $k \ge 1$.}
\end{align*}
The  {\it elephant random walk remembering the very recent past} (ERWVRP) $\{T_n : n \geq 0 \}$ with memory parameter $p$ is defined by
\[ T_0  := 0, \quad \mbox{and} \quad T_n = \sum_{k=1}^n X_k \mbox{ for $n \geq 1$.} \]

\begin{theorem} \label{thm:CorrRWCLTLIL} Let $p \in (0,1)$. For the ERWVRP $\{T_n : n \geq 0 \}$ with memory parameter $p$, we have
\begin{align}
\lim_{n \to \infty} P\left( \dfrac{T_n}{\sqrt{n}} \leq y\sqrt{\dfrac{p}{1-p}} \right)=\frac{1}{\sqrt{2\pi}}\int_{-\infty}^{y}e^{-t^{2}/2}dt \quad \mbox{for $y \in \mathbb{R}$}, \label{eq:CorrRWCLT}
\end{align}
and
\begin{align}
\limsup_{n \to \infty} \dfrac{T_n}{\sqrt{2n\log \log n}} = \sqrt{\dfrac{p}{1-p}} \quad \mbox{$P$-a.s.} \label{eq:CorrRWLIL}
\end{align}
\end{theorem}

The central limit theorem (CLT) in \eqref{eq:CorrRWCLT} is proved in Gut and Stadtm\"{u}ller \cite[Theorem 7.1]{GutStadtmuller21JAP}, where a CLT for uniformly mixing Markov chains is used (see also \cite{BenArietal21Brazil}).  A completely different proof of \eqref{eq:CorrRWCLT}, based on a strong structural similarity between the ERWVRP and quantum walks, is given in Konno \cite[Theorem 3.2]{Konno09CorrRW}. 
A proof of the law of the iterated logarithm (LIL) in \eqref{eq:CorrRWLIL} is found in Allaart \cite[Proposition 6.3]{Allaart09JMSJ}. For the sake of completeness, we will sketch a proof of both limit theorems in Subsection \ref{sub:ProofERWRPCLTLIL} below.

Let $\{a_k : k \geq 1 \}$ be a (deterministic) real sequence, and define
\[ S_0  := 0, \quad \mbox{and} \quad S_n = \sum_{k=1}^n a_kX_k  \mbox{ for $n \geq 1$.} \]
We call $\{S_n : n \geq 0 \}$  the {\it ERWVRP with variable step length}. 
The next theorem gives a necessary and sufficient condition for the localization (convergence) of $\{S_n : n \geq 0\}$.

\begin{theorem} \label{thm:CorrRWRademacherKhintchineKolmogorov} 
Let $p \in (0,1)$ and $\{a_k : k \geq 1 \}$ be a real sequence. The following hold for the ERWVRP with variable step length $\{S_n : n \geq 0 \}$ with parameter $p$: \\
\noindent (i) If $\displaystyle \sum_{k=1}^{\infty} (a_k)^2<+\infty$ then there exists a square-integrable random variable $S$ such that $S_n$ converges to $S$  as $n \to \infty$ $P$-a.s. and in $L^2$. \\
\noindent (ii) If $\displaystyle \sum_{k=1}^{\infty} (a_k)^2=+\infty$ then $\{S_n\}$ diverges $P$-a.s.
\end{theorem}

One of the authors \cite{Nakano25+} studied the limiting behaviour of the standard ERW with step length $a_k=k^{-\gamma}$ with $\gamma>0$, and showed that 
$\gamma>1/2$ is not sufficient for localization if the memory parameter $p>3/4$ --- strong memory effects prevent localization.  This contrasts with Theorem \ref{thm:CorrRWRademacherKhintchineKolmogorov}, which says that even if $p$ is very close to $1$, the localization criterion is exactly the same as that of the memoryless case $p=1/2$.

\subsection{Applications to the Takagi--van der Waerden class functions} About the oscillation of the Takagi--van der Waerden functions $f_r(x)$, 
Allaart and Kawamura \cite[p.32]{AllaartKawamura11/12} pointed out that for {\it fixed} $x \in [0,1]$,
\begin{align*}
-1 \leq \liminf_{h \to 0}  \frac{f_{r}(x+h)-f_{r}(x)}{h\log_{r}(1/|h|)} \leq \limsup_{h \to 0}  \frac{f_{r}(x+h)-f_{r}(x)}{h\log_{r}(1/|h|)} \leq 1.
\end{align*}
As an application of Theorem \ref{thm:CorrRWCLTLIL}, we prove the following theorem about the local modulus of the continuity of $f_r(x)$,
which extend the results obtained by K\^{o}no \cite{Kono87ActaMathHungar} and Gamkrelidze \cite{Gamkrelidze90} for $r=2$.

\begin{theorem}\label{thm:main_theorem}
        We consider the Takagi--van der Waerden functions \(f_{r}(x)\) defined by \eqref{eq:TakagivanderWaerdenfunction}.
        
            \noindent (1) If \(r\) is even, then
                \begin{align}
                    &\lim_{h\to 0} \mu\left(\left\{ x\in [0,1) : \frac{f_{r}(x+h)-f_{r}(x)}{h\sqrt{\log_{r}(1/|h|)}} \le y \right\} \right)=\int_{-\infty}^{y}\frac{1}{\sqrt{2\pi}}e^{-t^{2}/2}dt, \label{eq:CLTeven}\\
                    &\limsup_{h\to 0}  \frac{f_{r}(x+h)-f_{r}(x)}{h\sqrt{2\log_{r}(1/|h|)\log\log\log_{r}(1/|h|)}}=1 \quad \text{for $\mu$-a.e.}~ x. \label{eq:LILeven}
                \end{align}
                (2) If \(r\) is odd, then
                \begin{align}
                    &\lim_{h\to 0}\mu\left(\left\{ x\in [0,1) : \frac{f_{r}(x+h)-f_{r}(x)}{h\sqrt{\log_{r}(1/|h|)}}\le y\sqrt{\frac{r+1}{r-1}} \right\} \right)=\int_{-\infty}^{y}\frac{1}{\sqrt{2\pi}}e^{-t^{2}/2}dt, \label{eq:CLTodd} \\
                    &\limsup_{h\to 0} \frac{f_{r}(x+h)-f_{r}(x)}{h\sqrt{2\log_{r}(1/|h|)\log\log\log_{r}(1/|h|)}}=\sqrt{\frac{r+1}{r-1}} \quad \text{for $\mu$-a.e.}~ x. \label{eq:LILodd}
                \end{align}
    \end{theorem}
    
We mention two related results. Han, Schied, and Zhang \cite{HanSchiedZhang21} showed that the $\Phi$-variation of $f_r$ depends on the parity of $r$, by using a process essentially the same as the ERWVRP.
The Takagi--van der Waerden function $f_r(x)$ satisfies the functional equation
\begin{align} \label{eq:TvdWfunctionalEq}
f_r(rx) - r \cdot f_r(x) = -r \cdot d(x).
\end{align}
de Lima and Smania \cite{deLimaSmania19} obtained a similar result as Theorem \ref{thm:main_theorem} above (and much more) about the continuous function which is a unique bounded solution of the twisted cohomological equation, a generalization of the functional equation \eqref{eq:TvdWfunctionalEq}. Unfortunately their theory is not applicable to $f_r(x)$ itself, since $d(x) \notin C^{1+\varepsilon}$. We prove Theorem   \ref{thm:main_theorem} by establishing a suitable approximation of $f_r(x+h)-f_r(x)$ by the ERWVRP.
 
The next theorem gives a classification of the differentiability properties of Takagi--van der Waerden class  functions, and is proved by an application of Theorem \ref{thm:CorrRWRademacherKhintchineKolmogorov}.

 \begin{theorem} \label{thm:DIffTakagivdW} Let $\{ a_k : k \geq 1 \}$ be a real sequence satisfying \eqref{def:TakagivdWclassl1}, and consider the continuous function $f_{r,a}(x)$ defined by \eqref{def:TakagivdWclass}. \\
(i) If $\displaystyle \sum_{k=1}^{\infty} (a_k)^2<+\infty$, then  $f_{r,a}(x)$ is absolutely continuous, and hence differentiable for $\mu$-a.e. $x$. \\
(ii) If $\displaystyle \sum_{k=0}^{\infty} (a_k)^2=+\infty$ but $\displaystyle \lim_{k \to \infty} a_k=0$, then $f_{r,a}(x)$ is differentiable on a set of continuum and the range of the derivative is $\mathbb{R}$, but $f_{r,a}(x)$ is nondifferentiable for $\mu$-a.e. $x$. \\
(iii) If $\displaystyle \limsup_{k \to \infty} |a_k| >0$, then $f_{r,a}(x)$ has nowhere finite derivative. 
\end{theorem}

Ferrera, G\'{o}mez-Gil and Llorente \cite{FerreraGomezGil20Diff,FerreraGomezGil25} explored the differentiability properties of the generalized Takagi class functions, which contains the Takagi--van der Waerden class functions as an important subclass. Gathering the results in \cite{FerreraGomezGil20Diff,FerreraGomezGil25}, Theorem \ref{thm:DIffTakagivdW} for even $r$ can be obtained (see Section \ref{sec:ProofTheoremDIffTakagivdW} below for more details).
Combining their theory with our Theorem \ref{thm:CorrRWRademacherKhintchineKolmogorov}, we can establish Theorem \ref{thm:DIffTakagivdW} for odd $r$ also.

\section{Proofs of limit theorems for the ERWVRP}
\label{Sec:ProofERWRP}

\subsection{Preliminaries and a proof of Theorem \ref{thm:CorrRWCLTLIL}}
\label{sub:ProofERWRPCLTLIL}

Let $Q$ be the transition probability matrix for $\{X_k : k \geq 1 \}$, i.e.
\[ Q=\begin{pmatrix} p & 1-p \\ 1-p & p \\ \end{pmatrix} = \dfrac{1}{2} \begin{pmatrix} 1 & 1 \\ 1 & 1 \\ \end{pmatrix} +  \left(p-\dfrac{1}{2}\right) \cdot \begin{pmatrix} 1 & -1 \\ -1 & 1 \\ \end{pmatrix} =: Q_1+Q_2. \]
Since $(Q_1)^2=Q_1$, $(Q_2)^2=(2p-1) Q_2$ and $Q_1Q_2=Q_2Q_1=O$, we have
\begin{align}
Q^m =  \dfrac{1}{2} \begin{pmatrix} 1 & 1 \\ 1 & 1 \\ \end{pmatrix} + \dfrac{(2p-1)^m}{2} \cdot \begin{pmatrix} 1 & -1 \\ -1 & 1 \\ \end{pmatrix}\quad \mbox{for $m \in \mathbb{N}$}, \label{eq:Q^mExplicit}
\end{align}
which implies that
\[ |P(X_{n+m} = j \mid X_n=i) - P(X_n=i)| = \dfrac{|2p-1|^m}{2} =: \phi(m) \]
for $n \in \mathbb{Z}_+$, $m \in \mathbb{N}$, and $i,j \in \{+1,-1\}$.
As in  Stout \cite[Example 3.7.2]{Stout74}, we can see that $\{X_k : k \geq 1 \}$ is $\phi$-mixing:
For $A \in \sigma\{ X_k : 1 \leq k \leq n\}$ and $B \in \sigma \{ X_k : k \geq n+m \}$, 
\[
|P(A \cap B) - P(A)P(B)| \leq P(A) \cdot \phi(m).
\]

It is easy to see that $P(X_k=+1)=1/2$ for all $k \in \mathbb{N}$. By \eqref{eq:Q^mExplicit}, 
\[ P(X_k = X_{k+j}) = \dfrac{1+(2p-1)^j}{2}\quad \mbox{for $k,j \in \mathbb{N}$}, \]
which implies that
\begin{align}
 E[X_kX_{k+j}] =  P(X_k = X_{k+j}) -  P(X_k \neq X_{k+j}) = {(2p-1)^j} \quad \mbox{for $k,j \in \mathbb{N}$}.
 \label{eq:CorrMCnaive}
\end{align}
Since
\begin{align*}
E[(T_n)^2] &= \sum_{k=1}^n E[(X_k)^2]+2 \sum_{1 \leq k < \ell \leq n} E[X_kX_{\ell}] 
= n + 2 \sum_{j=1}^{n-1} (n-j) \cdot (2p-1)^j \\
&\sim n +  2n\sum_{j=1}^{\infty} (2p-1)^j = \dfrac{p}{1-p}n \quad \mbox{as $n \to \infty$,}
\end{align*}
Theorem \ref{thm:CorrRWCLTLIL} is a consequence of the CLT and the LIL for $\phi$-mixing sequences (see \cite[Theorem 27.5]{Billingsley79} and \cite[Theorem 5.4.4]{Stout74}).

\subsection{Proof of Theorem \ref{thm:CorrRWRademacherKhintchineKolmogorov}}

Throughout this section, $\{a_k : k \geq 1 \}$ denotes a deterministic real sequence. 

\begin{lemma} \label{lem:WeightedCorrelatedRW2ndMomentSpectral} Assume that  $p \in (0,1)$. For any $\{a_k : k \geq 1 \}$ and any integers $n>m \ge 1$, we have
\[\dfrac{1}{K(p)} \sum_{m<k\leq n} (a_k)^2 \leq E[(S_n-S_m)^2] \leq K(p) \sum_{m<k\leq n} (a_k)^2, \]
where 
\[ K(p)   :=  \max \left\{ \dfrac{p}{1-p}, \dfrac{1-p}{p} \right\} (\ge 1). \]
\end{lemma}

\begin{proof}
By \eqref{eq:CorrMCnaive},
\begin{align}
E[X_k X_{\ell}] = (2p-1)^{|k-\ell|}\quad \mbox{for $k,\ell \in \mathbb{N}$,}
\label{eq:2-Correlation} 
\end{align}
where $0^0=1$. 
According to Example (9.3.23) in Grimmett and Stirzaker \cite{GrimmettStirzaker20},
for $\alpha \in (-1,1)$, 
\begin{align}
\alpha^{|n|} 
= \dfrac{1}{2\pi} \int_{-\pi}^{\pi} e^{in\lambda}  \cdot \dfrac{1-\alpha^2}{1-2\alpha \cos \lambda + \alpha^2}\,d\lambda\quad \mbox{for $n \in \mathbb{Z}$.}
\label{eq:Example(9.3.23)GS20}
\end{align}
By \eqref{eq:2-Correlation} and \eqref{eq:Example(9.3.23)GS20}, 
\begin{align}
E[X_k X_{\ell}] =  \int_{-\pi}^{\pi} e^{i(k-\ell)\lambda}  \cdot  \rho(\lambda)\,d\lambda \quad \mbox{for $k,\ell \in \mathbb{N}$,}
\label{eq:2-CorrelationSpectral} 
\end{align}
where 
\[
\rho(\lambda)   :=  \dfrac{1}{2\pi} \cdot \dfrac{1-(2p-1)^2}{1-2(2p-1)\cos \lambda + (2p-1)^2}.
\]
We have
\begin{align*}
E[(S_n-S_m)^2] &= E\left[ \left(\sum_{m < k \leq n} a_k X_k \right)^2\right]  = \sum_{m<k,\ell \leq n} a_ka_{\ell} E[X_kX_{\ell}] \\
&= \int_{-\pi}^{\pi} \left( \sum_{m<k,\ell \leq n} a_k e^{ik \lambda} \cdot \overline{ a_{\ell} e^{i\ell \lambda}} \right) \cdot  \rho(\lambda)\,d\lambda \\
&=\int_{-\pi}^{\pi} \left| \sum_{m<k \leq n} a_k e^{ik\lambda} \right|^2 \,\rho(\lambda) \, d\lambda \quad \mbox{for $n > m \geq 1$.}
\end{align*}
Using
$\dfrac{1}{2\pi K(p)} \leq \rho(\lambda) \leq  \dfrac{1}{2\pi} \cdot K(p)$ 
for $\lambda \in [-\pi,\pi]$,
and a simple identity
\begin{align*}
\dfrac{1}{2\pi} \int_{-\pi}^{\pi} \left| \sum_{m<k \leq n} a_k e^{ik\lambda} \right|^2 \,\, d\lambda &= \dfrac{1}{2\pi} \int_{-\pi}^{\pi} \left( \sum_{m<k,\ell \leq n} a_k e^{ik \lambda} \cdot \overline{ a_{\ell} e^{i\ell \lambda}} \right)\,d\lambda \\
&= \sum_{m<k \leq n} (a_k)^2,
\end{align*}
we obtain the desired conclusion.
\end{proof}

\begin{proof}[Proof of Theorem \ref{thm:CorrRWRademacherKhintchineKolmogorov}] First we prove (ii). By Theorem 1 in Yoshihara \cite{Yoshihara78Kodai}, there exists a positive constant $C_1$ such that for any $\{a_k : k \geq 1\}$ and any integers $n>m \geq 1$, 
\begin{align}
E[(S_n-S_m)^4] = E\left[ \left(\sum_{m<k \leq n} a_k X_k \right)^4\right] \leq C_1 \left\{\sum_{m<k \leq n} (a_k)^2\right\}^2. \label{ineq:Yoshihara78Kodaim=4}
\end{align}
Assume that $\{S_n\}$ converges a.s.
By Egorov's theorem, we can find an event $A$ and a positive constant $M$ such that
\[ \mbox{$P(A^c) \leq \dfrac{1}{4C_1K(p)^2}$, and $|S_n-S_m| \leq M$ for any $n>m \geq 1$ a.s. on the event $A$.} \]
By Lemma \ref{lem:WeightedCorrelatedRW2ndMomentSpectral} and the Cauchy--Schwarz inequality,
\begin{align*}
\dfrac{1}{K(p)} \sum_{m<k \le n} (a_k)^2 &\leq E[(S_n-S_m)^2] \\
&\leq M^2 \cdot P(A) + \left\{E[(S_n-S_m)^4]\right\}^{1/2} \cdot P(A^c)^{1/2} \\
&\leq M^2 + \left\{\sqrt{C_1} \sum_{m<k \le n}  (a_k)^2 \right\} \cdot  \dfrac{1}{2K(p)\sqrt{C_1}}.
\end{align*}
Hence we have
\begin{align*}
\sum_{m<k \le n} (a_k)^2 \leq 2K(p) \cdot M^2 \quad \mbox{for any $n>m \geq 1$}.
\end{align*}
This implies that $\displaystyle \sum_{k=1}^{\infty} (a_k)^2<+\infty$. 
Conversely, if $\displaystyle \sum_{k=1}^{\infty} (a_k)^2=+\infty$ then
$P(\mbox{$\{S_n\}$ converges})<1$.
Since $\{X_k\}$ is $\phi$-mixing, its tail $\sigma$-field is trivial (see e.g. Section 2.5 in \cite{Bradley05survey}). Thus 
we have $P(\mbox{$\{S_n\}$ converges})=0$.

Now we turn to the proof of (i). Assume that $\displaystyle \sum_{k=1}^{\infty} (a_k)^2<+\infty$. By Lemma \ref{lem:WeightedCorrelatedRW2ndMomentSpectral}, we have
\begin{align*}
E[(S_n-S_m)^2] 
\leq K(p) \sum_{m < k \leq n} (a_k)^2 \leq  K(p) \sum_{k=N}^{\infty} (a_k)^2
\quad \mbox{for $n>m \geq N$,} 
\end{align*}
which implies that
$\displaystyle \lim_{N \to \infty} \sup_{n,m \geq N} E[(S_n-S_m)^2]  = 0$.
There exists a random variable $S$ with $E[S^2]<\infty$ and 
$\displaystyle \lim_{n \to \infty} E[(S_n-S)^2]   = 0$.
By Fatou's lemma and Lemma \ref{lem:WeightedCorrelatedRW2ndMomentSpectral}, 
\begin{align}
E[(S_n-S)^2] \leq \liminf_{N \to \infty}E[(S_n-S_N)^2]  \leq K(p) \sum_{k=n}^{\infty} (a_k)^2.
\label{eq:WeightedCorrelatedRW2ndBound}
\end{align}
We can find a subsequence $\{n_j : j \in \mathbb{N} \}$ of $\mathbb{N}$ such that 
\begin{align}
  \sum_{k=n_j}^{\infty} (a_k)^2 \leq \dfrac{1}{j^2} \quad \mbox{for each $j \in \mathbb{N}$}.
  \label{ineq:MainThmSubseq}
\end{align}
This together with \eqref{eq:WeightedCorrelatedRW2ndBound} implies that
\[ \sum_{j=1}^{\infty} E[(S_{n_j}-S)^2] \leq \sum_{j=1}^{\infty} K(p) \sum_{k=n_j}^{\infty} (a_k)^2 \leq K(p)  \sum_{j=1}^{\infty} \dfrac{1}{j^2}  <+\infty. \]
Beppo Levi's theorem yields that $\displaystyle \lim_{j \to \infty} S_{n_j} = S$ a.s. 
Theorem 1 in M\'{o}ricz \cite{Moricz76ZW} together with \eqref{ineq:Yoshihara78Kodaim=4} implies that there is a positive constant $C_2$ such that for any $\{a_k : k \geq 1 \}$ and $n>m \geq 1$, 
\begin{align}
E\left[ \left(\max_{m < \ell \leq n} |S_{\ell}-S_m| \right)^4\right] &\leq C_2 \cdot E[(S_n-S_m)^4] 
\leq C_1 C_2 \left\{\sum_{m<k\leq n} (a_k)^2\right\}^2.
\label{ineq:MainThmMaxMoricz}
\end{align}
Using the monotone convergence theorem together with  \eqref{ineq:MainThmMaxMoricz} and \eqref{ineq:MainThmSubseq}, we have
\begin{align*}
E\left[\sum_{j=1}^{\infty}  \left( \max_{n_j  < n \leq n_{j+1}} |S_n - S_{n_j}| \right)^4\right] &=
 \sum_{j=1}^{\infty} E\left[ \left( \max_{n_j  < n \leq n_{j+1}} |S_n - S_{n_j}| \right)^4\right] \\
 &\leq C_1C_2 \sum_{j=1}^{\infty} \left\{\sum_{n_j<k \leq n_{j+1}} (a_k)^2\right\}^2 \\
 &\leq C_1C_2\sum_{j=1}^{\infty} \dfrac{1}{j^4} < +\infty,
\end{align*} 
which implies that $\displaystyle \sum_{j=1}^{\infty}  \left( \max_{n_j  < n \leq n_{j+1}} |S_n - S_{n_j}| \right)^4<+\infty$ a.s., and hence
\[ \lim_{j \to \infty} \max_{n_j  < n \leq n_{j+1}} |S_n - S_{n_j}| = 0 \quad \mbox{a.s.} \]
Thus we have $\displaystyle \lim_{n \to \infty} S_n= S$ a.s.
\end{proof}

\section{Proof of Theorem \ref{thm:main_theorem}}
\label{sec:ProofTakagiA}



For each \(h\in (0,1/r)\), we can find a unique integer \(m=m(h)\) such that
\begin{equation}\label{eq:difinitionofh}
    \frac{1}{r^{m+1}}< h\le \frac{1}{r^{m}}\left(\leq \frac{1}{2r^{m-1}}\right).
\end{equation}
 Note that 
 \begin{equation} \label{asymp:m(h)}
 m(h)\sim \log_{r}(1/h) \quad \mbox{as \(h\downarrow 0\)}, 
\end{equation}
where \(a(h)\sim b(h)\) as \(h\downarrow 0\) means that \(\displaystyle \lim_{h\downarrow 0} \frac{a(h)}{b(h)}=1\). 
Our aim is to obtain the following approximation of the increment $f_{r}(x+h)-f_{r}(x)$ by the ERWVRP with memory parameter $p_r$, at time $m(h)$:
\begin{align}
    f_{r}(x+h)-f_{r}(x)&= h \cdot s_{m(h)}(x) +o\left( h\sqrt{m(h)} \right)\quad \mbox{as $h \downarrow 0$ for $\mu$-a.e. $x$}.
    \label{eq:GamkrelidzeDecomposition}
\end{align}
The approximation procedure for even $r$ is more or less similar to that for $r=2$, while we need a new idea for establishing a suitable approximation for odd $r$.

\begin{figure}[t]
\begin{center}
\begin{tabular}{cc}
\includegraphics[width=5.5cm]{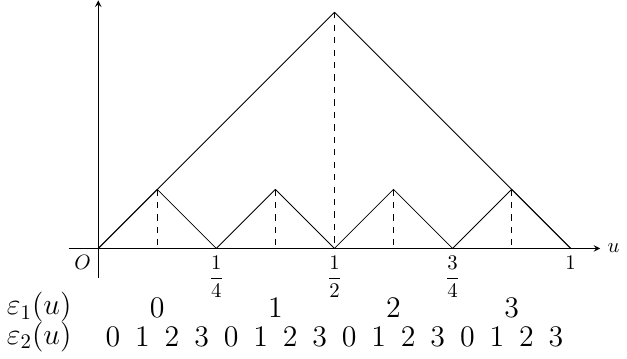}
&
\includegraphics[width=5.5cm]{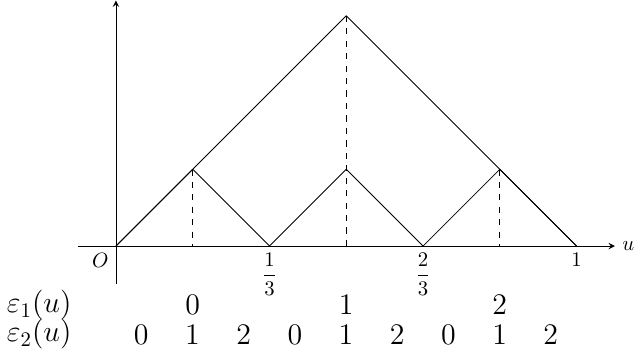}
\end{tabular}
\end{center}
\caption{The graphs of $\psi_1$ and $\psi_2$ for $r=4$ (left) and $r=3$ (right).}
\label{fig:TakagivdWphi1phi2}
\end{figure}

To obtain \eqref{eq:GamkrelidzeDecomposition}, we want to know for which values of $k$ the function $\psi_k$ is linear on $[x,x+h]$, and hence $\psi_k(x+h)-\psi_k(x)=\psi_k^+ (x) \cdot h$. By looking at the graph of $\psi_k$ (Figure \ref{fig:TakagivdWphi1phi2}), this is clearly the case if and only if 
\begin{align} \label{cond:psiklinearGraph}
\begin{aligned}
&\mbox{the points $x$ and $x+h$ belong to the same interval of the partition} \\
&\left\{ 0,\dfrac{1}{2r^{k-1}},\dfrac{2}{2r^{k-1}},\ldots,\dfrac{2r^{k-1}-1}{2r^{k-1}},1\right\}.
\end{aligned}
\end{align}
Here we consider $f_r(x)$ to be defined also to the right of $1$ by periodicity.

\subsection{Proof of Theorem \ref{thm:main_theorem} (1)}


For \(x\in [0,1)\) and  \(h\in (0,1/r)\), we consider \(r\)-ary expansions of \(x\) and \(x+h\):
\begin{equation}
    x=\sum_{k=0}^{\infty} \dfrac{{\varepsilon}_{k}}{r^k}, \quad \mbox{and} \quad x+h=\sum_{k=0}^{\infty} \dfrac{{\varepsilon}_{k}'}{r^k},
\end{equation}
where $\varepsilon_0=0$, $\varepsilon'_0 \in \{0,1\}$, and \({\varepsilon}_{k},{\varepsilon}_{k}' \in \{0,1,\dots,r-1\}\) for $k \in \mathbb{N}$. For $r$-adic rationals, we choose the representation ending in all zeros.

When $r$ is even, sufficient for \eqref{cond:psiklinearGraph} is that
\begin{align} \label{cond:psiklinearEVEN}
&\mbox{$x$ and $x+h$ have the same base-$r$ digits up to $k$-th,}
\end{align}
in the sense that $\varepsilon_\ell = \varepsilon'_\ell$ for all $\ell=0,1,\ldots,k$ (see Figure \ref{fig:TakagivdWphi1phi2}).

  Let
\begin{equation}
    k_{0}=k_{0}(x,h) := 
    \begin{cases}
        \max \{ k\in  \mathbb{Z}_+ : {\varepsilon}_{0}={\varepsilon}_{0}',\dots,{\varepsilon}_{k}={\varepsilon}_{k}' \} & \text{if $\varepsilon_{0}=\varepsilon_{0}'$},\\
        -1 & \text{if $\varepsilon_{0}\neq \varepsilon_{0}'$}.
    \end{cases}
\end{equation}
Note that $-1 \le k_{0} \le m$: If \(k_{0}\ge m+1\) then we have
\begin{equation*}
        h=\sum_{k=m+2}^{\infty}\frac{\varepsilon_{k}' - \varepsilon_{k}}{r^{k}}\le (r-1)\sum_{k=m+2}^{\infty}\frac{1}{r^{k}}=\frac{1}{r^{m+1}},
\end{equation*}
a contradiction. By the above observation, we have 
\begin{equation}
\label{eq:LinearEvenk0}
\psi_k(x+h)-\psi_k(x)=\psi_k^+ (x) \cdot h \quad \mbox{if $k \leq k_0(x,h)$.}
\end{equation}

For fixed $h \in (0,1/r)$, we can regard $k_0(x,h)$ as a $\{-1,0,1,\ldots,m(h)\}$--valued random variable. The following lemma shows that \(k_{0}(x,h)\) is not too far from \(m(h)\) with probability one.

\begin{lemma}\label{lem:p-k0convergentzero}
    $\displaystyle  \limsup_{h \downarrow 0} \dfrac{m(h)-k_{0}(x,h)}{2\log_r m(h)} \leq 1$
  for $\mu$-a.e. $x$. 
\end{lemma}

\begin{proof} First we show that
\begin{align} \label{ineq:m-k0tail}
\mu(m-k_0 \geq j) \le r^{-(j-1)} \quad \mbox{for $j \in \mathbb{N}$.}
\end{align}
This is obvious for $j > m+1$. For $1 \leq j \leq m$, if $m-k_0 \geq j$ then it must be the case that $\varepsilon_{k_0+1} = \cdots = \varepsilon_{m-1} = r-1$, so the base-$r$ expansion of $x$ must have at least $j-1$ digits $r-1$ in a row. For fixed $h$, this happens with probability $r^{-(j-1)}$. Thus, $\mu(m-k_0 \geq j) \le r^{-(j-1)}$. 
Finally, by  \eqref{eq:difinitionofh}, we have
\begin{align*}
\mu(m-k_0 \geq m+1) &= \mu(\{ x \in [0,1) : k_0(x,h)=-1\}) \\
&= \mu\bigl([1-h,1)\bigr) = h \leq r^{-m} = r^{-\{(m+1)-1\}}.
\end{align*}
This completes the proof of \eqref{ineq:m-k0tail}.

Noting that $m\left(r^{-\ell}\right)=\ell$ for each $\ell \in \mathbb{N}$, 
\begin{equation*}
 \sum_{\ell=1}^{\infty} \mu\left(\frac{m(r^{-\ell})-k_{0}(x,r^{-\ell})}{2\log_r m(r^{-\ell})}>1 \right)<+\infty.
\end{equation*}
By the Borel-Cantelli lemma we have $\displaystyle \limsup_{\ell \to \infty} \dfrac{\ell - k_0(x,r^{-\ell})}{2\log_r m(r^{-\ell})} \leq 1$ for $\mu$-a.e.  $x$. 
Take any positive sequence $\{ h_n \}$ with $\displaystyle \lim_{n \to \infty} h_n=0$. For each $n$, we can find a positive integer $\ell(n)$ satisfying $r^{-\ell(n)-1} < h_n \leq r^{-\ell(n)}$. Noting that $k_0(x,h)$ is nonincreasing in $h$ for fixed $x$, we have 
\[ 
 \dfrac{m(h_n)- k_0(x,h_n)}{2\log_r m(h_n)} \leq  \dfrac{\ell(n) -  k_0\left(x,r^{-\ell(n)}\right)}{2\log_r m(r^{-\ell(n)})} \quad \mbox{for each $n$}. \]
Since $\displaystyle \lim_{n \to \infty} \ell(n)=+\infty$, we have $\displaystyle \limsup_{n \to \infty}  \dfrac{m(h_n)- k_0(x,h_n)}{2\log_r m(h_n)} \leq 1$ $\mu$-a.e. $x$, and hence $\displaystyle \limsup_{h \downarrow 0}  \dfrac{m(h)- k_0(x,h)}{2\log_r m(h)} \leq 1$ $\mu$-a.e. $x$. This completes the proof. 
\end{proof}

Now we prove \eqref{eq:GamkrelidzeDecomposition}. By \eqref{eq:LinearEvenk0},
\begin{align}
    f_{r}(x+h)-f_{r}(x)&= h \cdot  s_{m(h)}(x)  +\sum_{k=k_{0}+1}^{m(h)} \{\psi_k(x+h)-\psi_k(x) - h \cdot \psi_{k}^{+}(x)\} \notag \\
    &\quad +  \sum_{k=m(h)+1}^{\infty} \{\psi_k(x+h)-\psi_k(x) \} .
    \label{eq:GamkrelidzeDecompositionVar}
\end{align}
Since $\psi_k$ is Lipschitz continuous and $|\psi_k^+(x)| \leq 1$, 
\begin{align*}
\left| \sum_{k=k_{0}+1}^{m(h)} \{\psi_k(x+h)-\psi_k(x) - h \cdot \psi_{k}^{+}(x)\} \right| \leq 2h(m-k_0).
\end{align*}
By \eqref{eq:difinitionofh}, we have
\begin{equation*}\label{eq:sigma3inequality}
    \left| \sum_{k=m(h)+1}^{\infty} \{\psi_k(x+h)-\psi_k(x) \} \right| \le \sum_{k=m+1}^{\infty}\frac{1}{2r^{k-1}}=\frac{r^{2}}{r-1} \cdot \frac{1}{r^{m+1}}\le \frac{r^{2}}{r-1}h.
\end{equation*}
Those estimates together with Lemma \ref{lem:p-k0convergentzero} and Eq. \eqref{asymp:m(h)} imply that the second and the third terms in the right hand side of \eqref{eq:GamkrelidzeDecompositionVar} are $o\left( h\sqrt{m(h)} \right)$ as  \(h\downarrow 0\). This completes the proof of  \eqref{eq:GamkrelidzeDecomposition}.

Using the CLT for the one-dimensional symmetric simple random walk (Eq. \eqref{eq:CorrRWCLT} in Theorem \ref{thm:CorrRWCLTLIL}, with $p=1/2$), Eq. \eqref{asymp:m(h)} and Slutsky's lemma, 
\begin{equation*}\label{eq:SumsofPsiConvergenceinDistribution}
    \lim_{h\downarrow 0} \mu\left(\left\{ x\in [0,1) : \frac{s_{m(h)}(x)}{\sqrt{\log_{r}(1/h)}}\le y \right\} \right) = \int_{-\infty}^{y}\frac{1}{\sqrt{2\pi}}e^{-t^{2}/2}dt\quad \mbox{for $y \in \mathbb{R}$.}
\end{equation*}
By  \eqref{eq:GamkrelidzeDecomposition} and another application of Slutsky's lemma, we have
\begin{align*}
                    \lim_{h\downarrow 0} \mu\left(\left\{ x\in [0,1) : \frac{f_{r}(x+h)-f_{r}(x)}{h\sqrt{\log_{r}(1/h)}} \le y \right\} \right)=\int_{-\infty}^{y}\frac{1}{\sqrt{2\pi}}e^{-t^{2}/2}dt \quad \mbox{for $y \in \mathbb{R}$.}
 \end{align*}
By considering $f_r(1-x)$, we can obtain the corresponding result for $h \uparrow 0$. This completes the proof of the CLT \eqref{eq:CLTeven}.

Using the LIL for the one-dimensional symmetric simple random walk (Eq. \eqref{eq:CorrRWLIL} in Theorem \ref{thm:CorrRWCLTLIL}, with $p=1/2$), we can prove the LIL \eqref{eq:LILeven} in a similar way as above.

\subsection{Proof of Theorem \ref{thm:main_theorem} (2)}

In view of Figure \ref{fig:TakagivdWphi1phi2}, we can see that the condition \eqref{cond:psiklinearGraph} is satisfied if and only if 
\begin{quote}
$x$ and $x+h$ have the same base-$r$ digits up to $(k-1)$-th, and also
$x + \dfrac{1}{2r^{k-1}}$ and $x + h + \dfrac{1}{2r^{k-1}}$ have the same base-$r$ digits up to $(k-1)$-th.
\end{quote}
When $r$ is odd, the second part of the condition is equivalent to 
 \begin{align*} 
&\mbox{$x+\dfrac{1}{2}$ and $x+h+\dfrac{1}{2}$ having the same base-$r$ digits up to $(k-1)$-th.}
\end{align*}
To see this, note that
\begin{align*}
\left\lfloor r^{k-1} \left( x+\dfrac{1}{2} \right) \right\rfloor &= \left\lfloor r^{k-1} \left( x+\dfrac{1}{2r^{k-1}} \right) + \dfrac{r^{k-1}-1}{2} \right\rfloor \\
&= \left\lfloor r^{k-1} \left( x+\dfrac{1}{2r^{k-1}} \right) \right\rfloor +  \dfrac{r^{k-1}-1}{2}, 
\intertext{where we used $r^{k-1}-1$ is even, and similarly}
\left\lfloor r^{k-1} \left( x+h+\dfrac{1}{2} \right) \right\rfloor  &= \left\lfloor r^{k-1} \left( x+h+\dfrac{1}{2r^{k-1}} \right) \right\rfloor +  \dfrac{r^{k-1}-1}{2}.
\end{align*}
Now the equivalence readily follows.

We introduce
\begin{align*}
\widehat{k}_0 = \widehat{k}_0(x,h)   :=   k_0\left( x+ \dfrac{1}{2},h \right), 
\end{align*}
and $k_0 \wedge \widehat{k}_0=(k_0 \wedge \widehat{k}_0)(x,h)  :=  \min\{ k_0(x,h),\widehat{k}_0(x,h) \}$.
Then we have 
\begin{equation}
\label{eq:LinearOddk0Hatk0}
\psi_k(x+h)-\psi_k(x)=\psi_k^+ (x) \cdot h \quad \mbox{if $k \leq (k_0 \wedge \widehat{k}_0)(x,h) $.}
\end{equation}

By the next lemma, when $r$ is odd, we can use $(k_{0}\wedge \widehat{k}_{0})(x,h)$ in place of $k_0(x,h)$.

\begin{lemma}\label{lem:p-k0hatk0convergentzero} 
$\displaystyle  \limsup_{h \downarrow 0} \dfrac{m(h)-(k_{0}\wedge \widehat{k}_{0})(x,h)}{2\log_r m(h)} \leq 1$ as $h \downarrow 0$ for $\mu$-a.e. $x$.    
 \end{lemma}

\begin{proof}
Since the Lebesgue measure is translation-invariant, it follows that $k_0$ and $\widehat{k}_0$ have the same distribution (because we are extending the definition of $f_r$ to the right). By \eqref{ineq:m-k0tail}, we have
\begin{align*}
\mu\left( m - k_{0}\wedge \widehat{k}_{0} \geq j \right) &\leq \mu(m-k_0 \geq j) + \mu\left(m-\widehat{k}_0 \geq j \right) \\
&=2\mu(m-k_0 \geq j) \le 2r^{-(j-1)}\quad \mbox{for $j \in \mathbb{N}$.}
\end{align*}
The rest of the proof is quite similar to that of Lemma \ref{lem:p-k0convergentzero}.
\end{proof}

With Lemma \ref{lem:p-k0hatk0convergentzero}, we can prove Theorem \ref{thm:main_theorem} (2) in a similar way as Theorem \ref{thm:main_theorem} (1). Note that $\sqrt{\dfrac{p_r}{1-p_r}} = \sqrt{\dfrac{r+1}{r-1}}$.

\section{Proof of Theorem \ref{thm:DIffTakagivdW}}
\label{sec:ProofTheoremDIffTakagivdW} 

The main theorem in Ferrera, G\'{o}mez-Gil and Llorente \cite[Theorem 1]{FerreraGomezGil25} can be applied to $f_{r,a}(x)$ for all $r \ge 2$, and $f_{r,a}(x)$ has nowhere finite derivative if and only if $\{a_k : k \geq 1\}$ {\it does not} satisfy $\displaystyle \lim_{k \to \infty} a_k=0$. 
On the other hand, consequences of several general results obtained by Ferrera and G\'{o}mez-Gil \cite{FerreraGomezGil20Diff} for $f_{r,a}(x)$ are summarized in \cite[Example 5.1]{FerreraGomezGil20Diff}: Notably, $f_{r,a}$ are differentiable at $\mu$-a.e. $x$ if and only if $\displaystyle \sum_{k=1}^{\infty} a_k \psi_k^+(x)$ converges $\mu$-a.e. $x$. Combined with Theorem \ref{thm:CorrRWRademacherKhintchineKolmogorov} with $p=p_r$, we obtain all the statements in Theorem \ref{thm:DIffTakagivdW}.

\begin{remark}
A result corresponding to our Theorem \ref{thm:CorrRWRademacherKhintchineKolmogorov} with $p=1/2$ is found in
Ferrera and G\'{o}mez-Gil \cite[Theorems 3.3 and 3.5]{FerreraGomezGil20Diff}, and hence Theorem \ref{thm:DIffTakagivdW} for even $r$ has already been settled by the theory of Ferrera, G\'{o}mez-Gil and Llorente \cite{FerreraGomezGil20Diff,FerreraGomezGil25}.
\end{remark}

\section*{Acknowledgement}
The authors thank the anonymous referee for detailed and helpful comments. They really appreciate the referee's suggestions which enable them to improve the presentation of Section 4 considerably.



\end{document}